\documentclass{amsart}
\usepackage{graphicx, color}
\usepackage{amscd}
\usepackage{amsmath,empheq}
\usepackage{amsfonts}
\usepackage{amssymb}
\usepackage{mathrsfs}

\usepackage[all]{xy}

\newtheorem{theorem}{Theorem}
\newtheorem{ex}{Example}
\newtheorem{lemma}[theorem]{Lemma}
\newtheorem{prop}[theorem]{Proposition}
\newtheorem{remark}{Remark}

\newtheorem{claim}{Claim}

\newenvironment{proof-sketch}{\noindent{\bf Sketch of Proof}\hspace*{1em}}{\qed\bigskip}

\everymath{\displaystyle}

\newcommand{\RR}{\mathbb R}
\newcommand{\NN}{\mathbb N}

\renewcommand{\leq}{\leqslant}

\renewcommand{\geq}{\geqslant}

\begin{document}

\title[Positive solutions for nonvariational Robin problems]{Positive solutions for nonvariational Robin problems}

\author[N.S. Papageorgiou]{Nikolaos S. Papageorgiou}
\address[N.S. Papageorgiou]{National Technical University, Department of Mathematics,
				Zografou Campus, Athens 15780, Greece}
\email{\tt npapg@math.ntua.gr}

\author[V.D. R\u{a}dulescu]{Vicen\c{t}iu D. R\u{a}dulescu}
\address[V.D. R\u{a}dulescu]{Faculty of Applied Mathematics,
AGH University of Science and Technology,
al. Mickiewicza 30, 30-059 Krak\'ow, Poland
 \& Institute of Mathematics ``Simion Stoilow" of the Romanian Academy, P.O. Box 1-764,
          014700 Bucharest, Romania}
\email{\tt vicentiu.radulescu@imar.ro}

\author[D.D. Repov\v{s}]{Du\v{s}an D. Repov\v{s}}
\address[D.D. Repov\v{s}]{Faculty of Education and Faculty of Mathematics and Physics, University of Ljubljana,
					 SI-1000 Ljubljana, Slovenia}
\email{\tt dusan.repovs@guest.arnes.si}


\begin{abstract}
We study a nonlinear Robin problem driven by the $p$-Laplacian and with a reaction term depending on the gradient (convection term). Using the theory of nonlinear operators of monotone-type and the asymptotic analysis of a suitable perturbation of the original equation, we show the existence of a positive smooth solution.
\end{abstract}

\maketitle

\section{Introduction}

Let $\Omega\subseteq\RR^N$ be a bounded domain with a $C^2$-boundary $\partial\Omega$. In this paper we deal with the following nonlinear Robin problem with gradient dependence:
\begin{equation}\label{eq1}
	\left\{\begin{array}{ll}
		-\Delta_pu(z)=f(z,u(z),Du(z))& \mbox{in}\ \Omega,\\
		\frac{\partial u}{\partial n_p}+\beta(z)|u|^{p-2}u=0&\mbox{on}\ \partial\Omega.
	\end{array}\right\}
\end{equation}

In this problem, $\Delta_p$ denotes the $p$-Laplacian differential operator defined by
$$\Delta_pu={\rm div}\,(|Du|^{p-2}Du)\ \mbox{for all}\ u\in W^{1,p}(\Omega),\ 1<p<\infty.$$

The reaction term $f(z,x,y)$ is gradient dependent (a convection term) and it is a  Ca\-ra\-th\'e\-o\-do\-ry function (that is, for all $(x,y)\in\RR\times\RR^N$ the mapping $z\mapsto f(z,x,y)$ is measurable and for almost all $z\in\Omega$ the map $(x,y)\mapsto f(z,x,y)$ is continuous). In the boundary condition, $\frac{\partial u}{\partial n_p}$ denotes the conormal derivative defined by extension of the map
$$C^1(\overline{\Omega})\ni u\mapsto |Du|^{p-2}\frac{\partial u}{\partial n}=|Du|^{p-2}(Du,n)_{\RR^N},$$
with $n(\cdot)$ being the outward unit normal on $\partial\Omega$. The boundary coefficient $\beta(\cdot)$ is nonnegative and it can be identically zero, in which case we recover the Neumann problem.

We are looking for positive solutions of problem (\ref{eq1}). The presence of the gradient in the reaction term precludes the use of variational methods. In this paper, our approach is based on the nonlinear operator theory and on the asymptotic analysis of a perturbation of problem (\ref{eq1}).

Positive solutions for elliptic problems with convection were obtained by de Figueiredo, Girardi and Matzeu \cite{3}, Girardi and Matzeu \cite{6} (semilinear equations driven by the Dirichlet Laplacian),  Ruiz \cite{13}, Faraci, Motreanu and Puglisi \cite{2}, and Huy, Quan and Khanh \cite{7} (nonlinear Dirichlet problems). For Neumann problems we refer to the works of Gasinski and Papageorgiou \cite{5}, and Papageorgiou, R\u adulescu and Repov\v{s} \cite{12}, where the differential operator is of the form ${\rm div}\,(a(u)Du)$. In all the above papers, the method of proof is different and 
 is based either on the fixed point theory (the Leray-Schauder alternative principle), on the iterative techniques, or on the method of upper-lower solutions.

\section{Mathematical Background and Hypotheses}

Let $X$ be a reflexive Banach space. We denote by $X^*$ its topological dual and by $\left\langle \cdot,\cdot\right\rangle$ the duality brackets for the dual pair $(X,X^*)$. Suppose that $V:X\rightarrow X^*$ is a nonlinear operator which is bounded (that is, it maps bounded sets to bounded sets) and everywhere defined. We say that $V(\cdot)$ is ``pseudomonotone", if the following property holds:
\begin{eqnarray*}
	&&u_n\stackrel{w}{\rightarrow}u\ \mbox{in}\ X,V(u_n)\stackrel{w}{\rightarrow}u^*\ \mbox{in}\ X^*\ \mbox{and}\ \limsup\limits_{n\rightarrow\infty}\left\langle V(u_n),u_n-u\right\rangle\leq 0\\
	&&\hspace{4cm}\Downarrow\\
	&&\hspace{1cm}u^*=V(u)\ \mbox{and}\ \left\langle V(u_n),u_n\right\rangle\rightarrow\left\langle V(u),u\right\rangle.
\end{eqnarray*}

Pseudomonotonicity is preserved by addition and any maximal monotone everywhere defined operator is pseudomonotone. Moreover, as is the case of maximal operators, pseudomonotone maps exhibit remarkable surjectivity properties.

\begin{prop}\label{prop1}
	If $V:X\rightarrow X^*$ is pseudomonotone and strongly coercive (that is,\\ $\frac{\left\langle V(u),u\right\rangle}{||u||} \rightarrow+\infty$ as $||u||\rightarrow\infty$), then $V$ is surjective.
\end{prop}

From the above remarks we see that if $A:X\rightarrow X^*$ is maximal monotone everywhere defined and $K:X\rightarrow X^*$ is completely continuous (that is, if $u_n\stackrel{w}{\rightarrow}u$ in $X$, then $K(u_n)\rightarrow K(u)$ in $X^*$), then $u\rightarrow V(u)=A(u)+K(u)$ is pseudomonotone.

A nonlinear operator $A:X\rightarrow X^*$ is said to be of type $(S)_+$, if the following property holds:
$$u_n\stackrel{w}{\rightarrow}u\ \mbox{in}\ X\ \mbox{and}\ \limsup\limits_{n\rightarrow\infty}\left\langle A(u_n),u_n-u\right\rangle\leq 0\Rightarrow u_n\rightarrow u\ \mbox{in}\ X.$$

For further details on these notions and related issues, we refer to Gasinski and Papageorgiou \cite{4}.

In the analysis of problem (\ref{eq1}) we will use the Sobolev space $W^{1,p}(\Omega)$, the Banach space $C^1(\overline{\Omega})$ and the boundary Lebesgue space $L^p(\partial\Omega)$.

We denote by $||\cdot||$  the norm of the Sobolev space $W^{1,p}(\Omega)$ defined by
$$||u||=[||u||^p_p+||Du||^p_p]^{1/p}\ \mbox{for all}\ u\in W^{1,p}(\Omega).$$

The Banach space $C^1(\overline{\Omega})$ is an ordered Banach space with positive (order) cone defined by
$$C_+=\{u\in C^1(\overline{\Omega}):u(z)\geq 0\ \mbox{for all}\ z\in\overline{\Omega}\}.$$

This cone has a nonempty interior given by
$${\rm int}\, C_+=\{u\in C_+:u(z)>0\ \mbox{for all}\ z\in\Omega,\left.\frac{\partial u}{\partial n}\right|_{\partial\Omega\cap u^{-1}(0)}<0\ \mbox{if}\ \partial\Omega\cap u^{-1}(0)\neq\emptyset\}.$$

This interior contains the open set
$$D_+=\{u\in C_+:u(z)>0\ \mbox{for all}\ z\in\overline{\Omega}\}.$$

In fact, $D_+$ is the interior of $C_+$ when $C^1(\overline{\Omega})$ is endowed with the $C(\overline{\Omega})$-norm topology.

On $\partial\Omega$ we consider the $(N-1)$-dimensional Hausdorff (surface) measure $\sigma(\cdot)$. Using this measure on $\partial\Omega$ we can define in the usual way the ``boundary" Lebesgue spaces $L^q(\Omega)$ ($1\leq q\leq\infty$). From the theory of Sobolev spaces, we know that there exists a unique continuous linear map $\gamma_0:W^{1,p}(\Omega)\rightarrow L^p(\partial\Omega)$, known as the ``trace map", such that $\gamma_0(u)=u|_{\partial\Omega}$ for all $u\in W^{1,p}(\Omega)\cap C(\overline{\Omega})$. So, the trace operator extends the notion of ``boundary values" to all Sobolev functions. We have
$${\rm im}\,\gamma_0= W^{\frac{1}{p'},p}(\partial\Omega)\left(\frac{1}{p}+\frac{1}{p'}=1\right)\ \mbox{and}\ {\rm ker}\,\gamma_0=W^{1,p}_0(\Omega).$$

The trace map is compact into $L^q(\partial\Omega)$ for all $q\in\left[1,\frac{(N-1)p}{N-p}\right)$ if $p<N$ and for all $q\geq 1$ if $N\leq p$. In the sequel, for the sake of notational simplicity, we drop the use of the trace map $\gamma_0$. All restrictions of Sobolev functions on $\partial\Omega$ are understood in the sense of traces.

Let $A:W^{1,p}(\Omega)\rightarrow W^{1,p}(\Omega)^*$ be the nonlinear operator defined by
$$\left\langle A(u),h\right\rangle=\int_{\Omega}|Du|^{p-2}(Du,Dh)_{\RR^N}dz\ \mbox{for all}\ u,h\in W^{1,p}(\Omega).$$
\begin{prop}\label{prop2}
	The operator $A:W^{1,p}(\Omega)\rightarrow W^{1,p}(\Omega)^*$ is bounded, continuous, monotone (hence also maximal monotone) and of type $(S)_+$.
\end{prop}

Given $x\in\RR$, we define $x^{\pm}=\max\{\pm x,0\}$. Then for $u\in W^{1,p}(\Omega)$ we set $u^{\pm}(\cdot)=u(\cdot)^{\pm}$. We have
$$u^{\pm}\in W^{1,p}(\Omega),\ u=u^+-u^-,\ |u|=u^++u^-.$$

Given a measurable function $g:\Omega\times\RR\times\RR^N\rightarrow\RR$ (for example, a Carath\'eodory function), we denote by $N_g(\cdot)$ the Nemitsky (superposition) map defined by
$$N_g(u)(\cdot)=g(\cdot,u(\cdot),Du(\cdot))\ \mbox{for all}\ u\in W^{1,p}(\Omega).$$

Evidently, $z\mapsto N_g(u)(z)$ is measurable. We denote by $|\cdot|_N$ the Lebesgue measure on $\RR^N$.

Consider the following nonlinear eigenvalue problem
\begin{equation}\label{eq2}
	\left\{\begin{array}{ll}
		-\Delta_p u(z)=\hat{\lambda}|u(z)|^{p-2}u(z)&\mbox{in}\ \Omega,\\
		\frac{\partial u}{\partial n_p}+\beta(z)|u|^{p-2}u=0&\mbox{on}\ \partial\Omega.
	\end{array}\right\}
\end{equation}

We make the following hypothesis concerning the boundary coefficient $\beta(\cdot)$:

\smallskip
$H(\beta):$ $\beta\in C^{0,\alpha}(\partial\Omega)$ with $\alpha\in(0,1)$ and $\beta(z)\geq 0$ for all $z\in\partial\Omega$.
\begin{remark}
	If $\beta\equiv 0$, then we recover the Neumann boundary condition.
\end{remark}

An ``eigenvalue" is a real number $\hat{\lambda}$ for which problem (\ref{eq2}) admits a nontrivial solution $\hat{u}\in W^{1,p}(\Omega)$, known as the ``eigenfunction" corresponding to the eigenvalue $\hat{\lambda}$. From Papageorgiou and R\u adulescu \cite{11} (see also Winkert \cite{14}), we have that
$$\hat{u}\in L^{\infty}(\Omega).$$

So, we can apply Theorem 2 of Lieberman \cite{8} and infer that
$$\hat{u}\in C^1(\overline{\Omega}).$$

From Papageorgiou and R\u adulescu \cite{10} we know that problem (\ref{eq2}) admits a smallest eigenvalue $\hat{\lambda}_1\in\RR$
 with the following properties:
\begin{itemize}
	\item $\hat{\lambda}_1\geq 0$, in fact $\hat{\lambda}_1=0$ if $\beta\equiv 0$ (Neumann problem) and $\hat{\lambda}_1>0$ if $\beta\not\equiv 0$.
	\item $\hat{\lambda}_1$ is isolated in the spectrum $\hat{\sigma}(p)$ of (\ref{eq2}) (that is, we can find $\epsilon>0$ such that $(\hat{\lambda}_1,\hat{\lambda}_1+\epsilon)\cap\hat{\sigma}(p)=\emptyset$).
	\item $\hat{\lambda}_1$ is simple (that is, if $\hat{u},\hat{v}\in C^1(\overline{\Omega})$ are eigenfunctions corresponding to $\hat{\lambda}_1$, then $\hat{u}=\xi\hat{v}$ for some $\xi\in\RR\backslash\{0\}$).
	\begin{equation}\label{eq3}
		\bullet\  \hat{\lambda}_1=\inf{\Bigg\{}\frac{||Du||^p_p+\int_{\partial\Omega}\beta(z)|u|^{p}d\sigma}{||u||^p_p}:u\in W^{1,p}(\Omega),u\neq 0{\Bigg\}}.\hspace{3.2cm}
	\end{equation}
\end{itemize}

The infimum in (\ref{eq3}) is realized on the corresponding one-dimensional eigenspace. From the above property it follows that the elements of this eigenspace do not change sign. Let $\hat{u}_1$ be the $L^p$-normalized (that is, $||\hat{u}_1||_p=1$) positive eigenfunction corresponding to $\hat{\lambda}_1$. We know that $\hat{u}_1\in C_+$. In fact, the nonlinear strong maximum principle (see, for example, Gasinski and Papageorgiou \cite[p. 738]{4}), implies that $\hat{u}_1\in D_+$. An eigenfunction $\hat{u}$ corresponding to an eigenvalue $\hat{\lambda}\not=\hat{\lambda}_1$, is necessary nodal (that is, sign changing). For more on the spectrum of (\ref{eq2}) we refer to Papageorgiou and R\u adulescu \cite{11}. The next lemma is an easy consequence of the above properties of the eigenpair $(\hat{\lambda}_1,\hat{u}_1)$ (see Mugnai and Papageorgiou \cite[Lemma 4.11]{9}).
\begin{lemma}\label{lem3}
	If $\vartheta\in L^{\infty}(\Omega),\ \vartheta(z)\leq\hat{\lambda}_1$ for almost all $z\in\Omega$, $\vartheta\not\equiv \hat{\lambda}_1$, then there exists $c_0>0$ such that
	$$||Du||^p_p+\int_{\partial\Omega}\beta(z)|u|^pd\sigma-\int_{\Omega}\vartheta(z)|u|^pdz\geq c_0||u||^p$$
	for all $u\in W^{1,p}(\Omega).$
\end{lemma}

Our hypotheses on the reaction term $f(z,x,y)$ are the following:

\smallskip
$H(f):$ $f:\Omega\times\RR\times\RR^N\rightarrow\RR$ is a Carath\'eodory function such that $f(z,0,y)=0$ for almost all $z\in\Omega$, for all $y\in\RR^N$, and
\begin{itemize}
	\item[(i)] $|f(z,x,y)|\leq a(z)[1+x^{p-1}+|y|^{p-1}]$ for almost all $z\in\Omega$, all $x\geq 0$, all $y\in\RR^N$, with $a\in L^{\infty}(\Omega)$;
	\item[(ii)] there exists a function $\vartheta\in L^{\infty}(\Omega)$ such that
	\begin{eqnarray*}
		&&\vartheta(z)\leq\hat{\lambda}_1\ \mbox{for almost all}\ z\in\Omega,\ \vartheta\not\equiv \hat{\lambda}_1,\\
		&&\limsup\limits_{x\rightarrow+\infty}\frac{f(z,x,y)}{x^{p-1}}\leq\vartheta(z)\ \mbox{uniformly for almost all}\ z\in\Omega,\ \mbox{and all}\ y\in\RR^N;
	\end{eqnarray*}
	\item[(iii)] for every $M>0$, there exists $\eta_M\in L^{\infty}(\Omega)$ such that $\eta_M(z)\geq\hat{\lambda}_1$ almost everywhere in $\Omega$, $\eta_M\not\equiv\hat{\lambda}_1$ and
	$$\liminf\limits_{x\rightarrow 0^+}\frac{f(z,x,y)}{x^{p-1}}\geq\eta_M(z)\ \mbox{uniformly for almost all}\ z\in\Omega,\ \mbox{and all}\ |y|\leq M.$$
\end{itemize}
\begin{remark}
	Since we are looking for positive solutions and the above hypotheses concern only the positive semiaxis $\RR_+=\left[0,+\infty\right)$, we may assume without loss of generality  that
	\begin{equation}\label{eq4}
		f(z,x,y)=0\ \mbox{for almost all}\ z\in\Omega,\ \mbox{all}\ x\leq 0, \ \mbox{and all}\ y\in\RR^N.
	\end{equation}
\end{remark}
\begin{ex}
	The following function satisfies hypotheses $H(f)$ (for the sake of simplicity we drop the $z$-dependence):
	$$f(x,y)=\left\{\begin{array}{ll}
		\eta x^{p-1}+x^{r-1}|y|^{p-1}&\mbox{if}\ 0\leq x\leq 1\\
		\vartheta x^{p-1}+(\eta-\vartheta)x^{q-1}+x^{\tau-1}|y|^{p-1}&\mbox{if}\ 1<x
	\end{array}\right.$$
	with $1<\tau,q<p<r<\infty$ and $\vartheta<\hat{\lambda}_1<\eta$.
\end{ex}

\section{Positive solution}

We introduce the following perturbation of $f(z,x,y)$:
$$\hat{f}(z,x,y)=f(z,x,y)+(x^+)^{p-1}.$$

Also, let $\epsilon>0$ and $e\in D_+$. We consider the following auxiliary Robin problem:
\begin{equation}\label{eq5}
	\left\{\begin{array}{ll}
		-\Delta_pu(z)+|u(z)|^{p-2}u(z)=\hat{f}(z,u(z),Du(z))+\epsilon e(z)&\mbox{in}\ \Omega,\\
		\frac{\partial u}{\partial n_p}+\beta(z)|u|^{p-2}u=0&\mbox{on}\ \partial\Omega.
	\end{array}\right\}
\end{equation}
\begin{prop}\label{prop4}
	If hypotheses $H(\beta),H(f)$ hold and $\epsilon>0$, then problem (\ref{eq5}) has a solution $u_{\epsilon}\in D_+$.
\end{prop}
\begin{proof}
	Let $N_{\hat{f}}$ be the Nemitsky map corresponding to the function $\hat{f}(z,x,y)$. We have $N_{\hat{f}}:W^{1,p}(\Omega)\rightarrow L^{p'}(\Omega)$ $\left(\frac{1}{p}+\frac{1}{p'}=1\right)$ (see hypothesis $H(f)(i)$). By Krasnoselskii's theorem (see, for example, Gasinski and Papageorgiou \cite[Theorem 3.4.4, p. 407]{4}) we deduce that
	\begin{equation}\label{eq6}
		N_{\hat{f}}(\cdot)\ \mbox{is continuous}.
	\end{equation}
	
	Also let $\psi_p:W^{1,p}(\Omega)\rightarrow L^{p'}(\Omega)$ be defined by
	$$\psi_p(u)(\cdot)=|u(\cdot)|^{p-2}u(\cdot).$$
	
	This map is bounded, continuous, monotone, hence also maximal monotone (recall that also $L^{p'}(\Omega)\hookrightarrow W^{1,p}(\Omega)^*$).
	
	Finally, let $\hat{A}:W^{1,p}(\Omega)\rightarrow W^{1,p}(\Omega)^*$ be defined by
	$$\left\langle \hat{A}(u),h\right\rangle=\left\langle  A(u),h\right\rangle+\int_{\partial\Omega}\beta(z)|u|^{p-2}uhd\sigma,$$
	where, as before,
	$$\left\langle A(u),h\right\rangle=\int_{\Omega}|Du|^{p-2}(Du,Dh)_{\RR^N}dz\quad \mbox{for all}\ u,h\in W^{1,p}(\Omega).$$
	
	Evidently, $\hat{A}(\cdot)$ is bounded, continuous, monotone, hence also maximal monotone.
	
	We introduce the operator $V:W^{1,p}(\Omega)\rightarrow W^{1,p}(\Omega)^*$ defined by
	$$V(u)=\hat{A}(u)+\psi_p(u)-N_{\hat{f}}(u)-\epsilon e.$$
	
	Clearly, $V(\cdot)$ is bounded.
	\begin{claim}
		$V(\cdot)$ is pseudomonotone.
	\end{claim}
	
	We need to show that the properties
	\begin{equation}\label{eq7}
		u_n\stackrel{ w}{\rightarrow} u\ \mbox{in}\ W^{1,p}(\Omega)\ \mbox{and}\ \limsup\limits_{n\rightarrow\infty}\left\langle V(u_n),u_n-u\right\rangle\leq 0
	\end{equation}
	imply that
	$$V(u_n)\stackrel{w}{\rightarrow} V(u)\ \mbox{in}\ W^{1,p}(\Omega)^*\ \mbox{and}\ \left\langle V(u_n),u_n\right\rangle\rightarrow\left\langle V(u),u\right\rangle.$$
	
	We have
	\begin{eqnarray}\label{eq8}
		&&\left\langle V(u_n),u_n-u\right\rangle\nonumber\\
		&=&\left\langle \hat{A}(u_n),u_n-u\right\rangle+\int_{\Omega}|u_n|^{p-2}u_n(u_n-u)dz-\int_{\Omega}\hat{f}(z,u_n,Du_n)(u_n-u)dz-\nonumber\\
		&&-\epsilon\int_{\Omega}e(u_n-u)dz.
	\end{eqnarray}
	
	Note that since $W^{1,p}(\Omega)\hookrightarrow L^p(\Omega)$ compactly, we have
	\begin{equation}\label{eq9}
		u_n\rightarrow u\ \mbox{in}\ L^p(\Omega).
	\end{equation}
	
	Also, we have
	$$\{|u_n|^{p-2}u_n\}_{n\geq 1}\subseteq L^{p'}(\Omega)\ \mbox{is bounded}.$$
	
	Hence, because of H\"older's inequality and (\ref{eq9}), we have
	\begin{equation}\label{eq10}
		\int_{\Omega}|u_n|^{p-2}u_n(u_n-u)dz\rightarrow 0\ \mbox{as}\ n\rightarrow\infty.
	\end{equation}
	
	Also, hypothesis $H(f)(i)$ implies that
	$$\{N_{\hat{f}}(u_n)\}_{n\geq 1}\subseteq L^{p'}(\Omega)\ \mbox{is bounded}.$$
	
	Therefore we also have
	\begin{equation}\label{eq11}
		\int_{\Omega}\hat{f}(z,u_n,Du_n)(u_n-u)dz\rightarrow 0\ \mbox{as}\ n\rightarrow\infty.
	\end{equation}
	
	Finally, we clearly have
	\begin{equation}\label{eq12}
		\int_{\Omega}e(u_n-u)dz\rightarrow 0\ \mbox{as}\ n\rightarrow\infty\ (\mbox{see (\ref{eq9})}).
	\end{equation}
	
	Thus, if in (\ref{eq8}) we pass to the limit as $n\rightarrow\infty$ and use (\ref{eq7}), (\ref{eq10}), (\ref{eq11}), and (\ref{eq12}) we obtain
	$$\limsup\limits_{n\rightarrow\infty}\left\langle \hat{A}(u_n),u_n-u\right\rangle\leq 0.$$
	
	By the compactness of the trace map, we have
	\begin{eqnarray*}
		&&\int_{\partial\Omega}\beta(z)|u_n|^{p-2}u_n(u_n-u)d\sigma\rightarrow 0,\\
		&\Rightarrow&\limsup\limits_{n\rightarrow\infty}\left\langle A(u_n),u_n-u\right\rangle\leq 0,\\
		&\Rightarrow&u_n\rightarrow u\ \mbox{in}\ W^{1,p}(\Omega)\ (\mbox{see Proposition \ref{prop2}}).
	\end{eqnarray*}
	
	On account of this convergence, we have
	\begin{eqnarray*}
		&&\psi_p(u_n)\rightarrow\psi_p(u)\ \mbox{and}\ N_{\hat{f}}(u_n)\rightarrow N_{\hat{f}}(u)\ \mbox{in}\ L^{p'}(\Omega)\ \mbox{as}\ n\rightarrow\infty\ \mbox{(see (\ref{eq6}))},\\
		&&\hat{A}(u_n)\rightarrow\hat{A}(u)\ \mbox{in}\ W^{1,p}(\Omega)^*\ \mbox{as}\ n\rightarrow\infty.
	\end{eqnarray*}
	
	So, we can finally assert
	 that
	\begin{eqnarray*}
		&&V(u_n)\rightarrow V(u)\ \mbox{in}\ W^{1,p}(\Omega)^*\ \mbox{and}\ \left\langle V(u_n),u_n\right\rangle\rightarrow\left\langle V(u),u\right\rangle,\\
		&\Rightarrow&V(\cdot)\ \mbox{is pseudomonotone}.
	\end{eqnarray*}
	
	This proves the claim.
	
	For all $u\in W^{1,p}(\Omega)$ we have
	\begin{eqnarray}\label{eq13}
		&&\left\langle V(u),u\right\rangle\nonumber\\ &=&||Du||^p_p+\int_{\partial\Omega}\beta(z)|u|^pd\sigma+||u^-||^p_p-\int_{\Omega}f(z,u,Du)udz-
\epsilon\int_{\Omega}eudz.
	\end{eqnarray}
	
	Hypotheses $H(f)(i),(ii)$ imply that given $\epsilon>0$, we can find $c_1=c_1(\epsilon)>0$ such that
	\begin{equation}\label{eq14}
		f(z,x,y)x\leq(\vartheta(z)+\epsilon)x^p+c_1\ \mbox{for almost all}\ z\in\Omega,\ \mbox{all}\ x\geq 0,\ \mbox{and all}\ y\in\RR^N.
	\end{equation}
	
  Using (\ref{eq14}) in (\ref{eq13}), we obtain
	\begin{eqnarray*}
		&&\left\langle V(u),u\right\rangle\\
		&\geq&||Du^-||^p_p+||u^-||^p_p+||Du^+||^p_p+\int_{\partial\Omega}\beta(z)(u^+)^pd\sigma-\int_{\Omega}\vartheta(z)(u^+)^pdz-\epsilon||u^+||^p\\
		&&-c_2||u||-c_1|\Omega|_N\ \mbox{for some}\ c_2>0,\\
		&\Rightarrow&\left\langle V(u),u\right\rangle\geq||u^-||^p+(c_0-\epsilon)||u^+||^p-c_2||u||-c_1|\Omega|_N\ (\mbox{see Lemma \ref{lem3}}).
	\end{eqnarray*}
	
	Choosing $\epsilon\in(0,c_0)$, we see that
	\begin{eqnarray}\label{eq15}
		&&\left\langle V(u),u\right\rangle\geq c_3||u||^p-c_4\ \mbox{for some}\ c_3,c_4>0,\nonumber\\
		&\Rightarrow&V(\cdot)\ \mbox{is strongly coercive (recall that $p>1$)}.
	\end{eqnarray}
	
	Then the claim and (\ref{eq15}) permit the use of Proposition \ref{prop1}. So, we can find $u_{\epsilon}\in W^{1,p}(\Omega),u_{\epsilon}\neq 0$ (since $e\neq 0$) such that
	\begin{eqnarray}\label{eq16}
		&&V(u_{\epsilon})=0\ \mbox{in}\ W^{1,p}(\Omega)^*\nonumber\\
		&\Rightarrow&\left\langle A(u_{\epsilon}),h\right\rangle+\int_{\partial\Omega}\beta(z)|u_{\epsilon}|^{p-2}u_{\epsilon}hd\sigma-\int_{\Omega}(u^-_{\epsilon})^{p-1}hdz\nonumber\\
		&&=\int_{\Omega}f(z,u_{\epsilon},Du_{\epsilon})hdz+\epsilon\int_{\Omega}ehdz\ \mbox{for all}\ h\in W^{1,p}(\Omega).
	\end{eqnarray}
	
	In (\ref{eq16}) we choose $h=-u^-_{\epsilon}\in W^{1,p}(\Omega)$ and use (\ref{eq4}) and hypothesis $H(\beta)$. Then
	\begin{eqnarray*}
		&&||Du^-_{\epsilon}||^p_p+||u^-_{\epsilon}||^p_p\leq 0\ (\mbox{recall that}\ e\in D_+),\\
		&\Rightarrow&u_{\epsilon}\geq 0,\ u_{\epsilon}\neq 0.
	\end{eqnarray*}
	
	Then from (\ref{eq16}) we have
	\begin{eqnarray}\label{eq17}
		&&\left\langle A(u_{\epsilon}),h\right\rangle+\int_{\partial\Omega}\beta(z)u^{p-1}_{\epsilon}hd\sigma=\int_{\Omega}f(z,u_{\epsilon},Du_{\epsilon})hdz+\epsilon\int_{\Omega}ehdz\ \mbox{for all}\ h\in W^{1,p}(\Omega)\nonumber\\
		&\Rightarrow&-\Delta_pu_{\epsilon}(z)=f(z,u_{\epsilon}(z),Du_{\epsilon}(z))+\epsilon e(z)\ \mbox{for almost all}\ z\in\Omega,\nonumber\\
		&&\frac{\partial u_{\epsilon}}{\partial n_p}+\beta(z)u^{p-1}_{\epsilon}=0\ \mbox{on}\ \partial\Omega\ (\mbox{see Papageorgiou and R\u adulescu \cite{10}}).
	\end{eqnarray}
	
	By Winkert \cite{14} and Papageorgiou and R\u adulescu \cite{11}, we have
	$$u_{\epsilon}\in L^{\infty}(\Omega).$$
	
	Applying Theorem 2 of Lieberman \cite{8}, we obtain
	$$u_{\epsilon}\in C_+\backslash\{0\}.$$
	
	Let $M=||u_{\epsilon}||_{C^1(\overline{\Omega})}$. Hypotheses $H(f)(i),(iii)$ imply that we can find $\hat{\xi}_M>0$ such that
	$$f(z,x,y)+\hat{\xi}_Mx^{p-1}\geq 0$$
	for almost all $z\in\Omega$, all $x\in[0,M]$, and all $|y|\leq M$.
	
	Using this in (\ref{eq17}), we have
	\begin{eqnarray*}
		&&\Delta_pu_{\epsilon}(z)\leq\hat{\xi}_Mu_{\epsilon}(z)^{p-1}\ \mbox{for almost all}\ z\in\Omega,\\
		&\Rightarrow&u_{\epsilon}\in D_+
	\end{eqnarray*}
	(by the nonlinear strong maximum principle, see \cite[p. 738]{4}).
\end{proof}

Next, we show that for some $\mu\in(0,1)$ and all $0<\epsilon\leq 1$, we have $u_{\epsilon}\in C^{1,\mu}(\overline{\Omega})$  and
$$\{u_{\epsilon}\}_{0<\epsilon\leq 1}\subseteq C^{1,\mu}(\overline{\Omega})\ \mbox{is bounded}.$$

Using this fact and letting $\epsilon\rightarrow 0^+$, we will generate a positive solution for problem (\ref{eq1}).
\begin{prop}\label{prop5}
If hypotheses $H(\beta),H(f)$ hold, then there exist $\mu\in(0,1)$ and $c^*>0$ such that for all $0<\epsilon\leq 1$ we have
$$u_{\epsilon}\in C^{1,\mu}(\overline{\Omega})\ \mbox{and}\ ||u_{\epsilon}||_{C^{1,\mu}(\overline{\Omega})}\leq c^*.$$
\end{prop}
\begin{proof}
	Let $\epsilon\in\left(0,1\right]$ and let $u_{\epsilon}\in D_+$ be a solution of (\ref{eq5}) produced in Proposition \ref{prop4}. We have
	\begin{eqnarray}\label{eq18}
		&&\left\langle A(u_{\epsilon}),h\right\rangle+\int_{\partial\Omega}\beta(z)u^{p-1}_{\epsilon}hd\sigma=\int_{\Omega}f(z,u_{\epsilon},Du_{\epsilon})hdz+\epsilon\int_{\Omega}ehdz\ \mbox{for all}\ h\in W^{1,p}(\Omega).
	\end{eqnarray}
	
	Hypothesis $H(f)(ii)$ implies that given $\epsilon>0$, we can find $M_1=M_1(\epsilon)>0$ such that
	\begin{equation}\label{eq19}
		f(z,x,y)x\leq(\vartheta(z)+\epsilon)x^p\ \mbox{for almost all}\ z\in\Omega,\ \mbox{all}\ x\geq M_1,\ \mbox{and all}\ y\in\RR^N.
	\end{equation}
	
	Also, hypothesis $H(f)(i)$ implies that
	\begin{eqnarray}\label{eq20}
		&&f(z,x,y)x\leq c_5(1+|y|^{p-1})\ \mbox{for almost all}\ z\in\Omega,\nonumber\\
		&& \mbox{all}\ 0\leq x\leq M_1,\ \mbox{all}\ y\in\RR^N,\ \mbox{some}\ c_5>0.
	\end{eqnarray}
	
	Then from (\ref{eq19}), (\ref{eq20}) and since $\vartheta\in L^{\infty}(\Omega)$, it follows that
	\begin{eqnarray}\label{eq21}
		&&f(z,x,y)x\leq(\vartheta(z)+\epsilon)x^p+c_6|y|^{p-1}+c_6\ \mbox{for almost all}\ z\in\Omega,\nonumber\\
		&& \mbox{all}\ x\geq 0,\ \mbox{all}\ y\in\RR^N,\ \mbox{and for some}\ c_5>0.
	\end{eqnarray}
	
	In (\ref{eq18}) we choose $h=u_{\epsilon}\in W^{1,p}(\Omega)$. We obtain
	\begin{eqnarray*} &&||Du_{\epsilon}||^p_p+\int_{\partial\Omega}\beta(z)u^p_{\epsilon}d\sigma\leq\int_{\Omega}[\vartheta(z)+\epsilon]u^p_{\epsilon}dz+c_7[||Du_{\epsilon}||^{p-1}_p+||u_{\epsilon}||+1]\\
		&&\mbox{ for some}\ c_7>0,\\
		&\Rightarrow&||Du_{\epsilon}||^p_p+\int_{\partial\Omega}\beta(z)u^p_{\epsilon}d\sigma-\int_{\Omega}\vartheta(z)u^p_{\epsilon}dz-\epsilon||u_{\epsilon}||^p\leq c_8[||u_{\epsilon}||^{p-1}+1]\ \mbox{for some}\ c_8>0,\\
		&\Rightarrow&[c_0-\epsilon]||u_{\epsilon}||^p\leq c_8[||u_{\epsilon}||^{p-1}+1]\ (\mbox{see Lemma \ref{lem3}}).
	\end{eqnarray*}
	
	Choosing $\epsilon\in(0,c_0)$, we infer that
	\begin{equation}\label{eq22}
		\{u_{\epsilon}\}_{0<\epsilon\leq 1}\subseteq W^{1,p}(\Omega)\ \mbox{is bounded}.
	\end{equation}
	
	From (\ref{eq18}) we have
	\begin{equation}\label{eq23}
		\left\{\begin{array}{ll}
			-\Delta_pu_{\epsilon}(z)=f(z,u_{\epsilon}(z),Du_{\epsilon}(z))+\epsilon e(z)\ \mbox{for almost all}\ z\in\Omega,&\\
			\frac{\partial u_{\epsilon}}{\partial n_p}+\beta(z)u^{p-1}_{\epsilon}=0\ \mbox{on}\ \partial\Omega&
		\end{array}\right\}
	\end{equation}
	(see Papageorgiou and R\u adulescu \cite{10}).
	
	From (\ref{eq22}), (\ref{eq23}) and Winkert \cite{14} (see also Papageorgiou and R\u adulescu \cite{11}), we see that we can find $c_9>0$ such that
	$$||u_{\epsilon}||_{\infty}\leq c_9\ \mbox{for all}\ 0<\epsilon\leq 1.$$
	
	Invoking Theorem 2 of Lieberman \cite{8}, we know that there exist $\mu\in(0,1)$ and $c^*>0$ such that
	$$u_{\epsilon}\in C^{1,\mu}(\overline{\Omega})\ \mbox{and}\ ||u_{\epsilon}||_{C^{1,\mu}(\overline{\Omega})}\leq c^*\ \mbox{for all}\ \epsilon\in\left(0,1\right].$$
This completes the proof.
\end{proof}
Now letting $\epsilon\rightarrow 0^+$, we will produce a positive solution for problem (\ref{eq1}).
\begin{theorem}\label{th6}
	If hypotheses $H(\beta),H(f)$ hold, then problem (\ref{eq1}) has a positive solution $\hat{u}\in D_+$.
\end{theorem}
\begin{proof}
	Let $\{\epsilon_n\}_{n\geq 1}\subseteq\left(0,1\right]$ and assume that $\epsilon_n\rightarrow 0^+$. We set $u_n=u_{\epsilon_n}$ for all $n\in\NN$. On account of Proposition \ref{prop5} and since $C^{1,\mu}(\overline{\Omega})$ is embedded compactly into $C^1(\overline{\Omega})$, by passing to a subsequence if necessary, we may assume that
	\begin{equation}\label{eq24}
		u_n\rightarrow \hat{u}\ \mbox{in}\ C^1(\overline{\Omega})\ \mbox{as}\ n\rightarrow\infty.
	\end{equation}
	
	Suppose that $\hat{u}=0$. Let $M=\sup\limits_{n\geq 1}||u_n||_{C^1(\overline{\Omega})}$. Hypothesis $H(f)(iii)$ implies that given $\epsilon>0$, we can find $\delta=\delta(\epsilon)>0$ such that
	\begin{equation}\label{eq25}
		f(z,x,y)\geq[\eta_M(z)-\epsilon]x^{p-1}\ \mbox{for almost all}\ z\in\Omega,\ \mbox{and all}\ 0\leq x\leq \delta,\ \mbox{all}\ |y|\leq M.
	\end{equation}	
	
	Consider the function $$R(\hat{u}_1,u_n)(z)=|D\hat{u}_1(z)|^p-|Du_n(z)|^{p-2}(Du_n(z),D\left(\frac{\hat{u}_1^p}{u_n^{p-1}}\right)(z))_{\RR^N}.$$
	
	By
	the nonlinear Picone
	 identity of Allegretto and Huang \cite{1}, we have
	\begin{eqnarray}\label{eq26}
		&0&\leq\int_{\Omega}R(\hat{u}_1,u_n)dz\nonumber\\
		&&=||D\hat{u}_1||^p_p-\int_{\Omega}|Du_n|^{p-2}(Du_n,D\left(\frac{\hat{u}_1^p}{u_n^{p-1}}\right))_{\RR^N}dz\nonumber\\
		&&=||D\hat{u}_1||^p_p-\int_{\Omega}(-\Delta_pu_n)\left(\frac{\hat{u}_1^p}{u_n^{p-1}}\right)dz+\int_{\partial\Omega}\beta(z)u_n^{p-1}\frac{\hat{u}^p_1}{u_n^{p-1}}d\sigma\nonumber\\
		&&\mbox{(by the nonlinear Green's identity, see Gasinski and Papageorgiou \cite[ p. 211]{4})}\nonumber\\
		&&=||D\hat{u}_1||^p_p+\int_{\partial\Omega}\beta(z)\hat{u}_1^{p-1}d\sigma-\int_{\Omega}f(z,u_n,Du_n)\frac{\hat{u}_1^p}{u_n^{p-1}}dz-\epsilon_n\int_{\Omega}e\frac{\hat{u}_1^p}{u_n^{p-1}}dz\nonumber\\
		&&(\mbox{see (\ref{eq23}) with}\ u_{\epsilon}\ \mbox{replaced by}\ u_n)\nonumber\\
		&&\leq\hat{\lambda}_1-\int_{\Omega}\eta_M(z)u_n^{p-1}\frac{\hat{u}_1^p}{u_n^{p-1}}dz+\epsilon\ \mbox{for all}\ n\geq n_0\nonumber\\
		&&(\mbox{see (\ref{eq25}), (\ref{eq24}) and recall that}\ \hat{u}=0\ \mbox{and}\ ||\hat{u}_1||_p=1)\nonumber\\
		&&=\hat{\lambda}_1-\int_{\Omega}\eta_M(z)\hat{u}_1^pdz+\epsilon\nonumber\\
		&&=\int_{\Omega}[\hat{\lambda}_1-\eta(z)]\hat{u}_1^pdz+\epsilon\ \mbox{for all}\ n\geq n_0\ (\mbox{recall that}\ ||\hat{u}_1||_p=1).
	\end{eqnarray}
	
	Let $\xi^*=\int_{\Omega}[\eta_M(z)-\hat{\lambda}_1]\hat{u}_1^pdz$. Since $\hat{u}_1\in D_+$, hypothesis $H(f)(iii)$ implies that
	$$\xi^*>0.$$
	
	Then from (\ref{eq26}) and by choosing $\epsilon\in(0,\xi^*)$ we have
	$$0\leq R(\hat{u}_1,u_n)<0\ \mbox{for all}\ n\geq n_0,$$
	a contradiction. So, $\hat{u}\neq 0$. Therefore, $\hat{u}\geq 0$ is a positive solution of (\ref{eq1}) and as before, via the nonlinear strong maximum principle, we have $\hat{u}\in D_+$.
\end{proof}

\medskip
{\bf Acknowledgements.} This research  was supported in part by  the  Slovenian  Research  Agency
grants P1-0292, J1-7025, J1-8131, and N1-0064. V.D.~R\u adulescu acknowledges the support through a grant  of the Romanian Ministry
of Research and Innovation, CNCS--UEFISCDI, project number PN-III-P4-ID-PCE-2016-0130,
within PNCDI III.


\begin{thebibliography}{99}

\bibitem{1} W. Allegretto, Y.X. Huang,  A Picone's identity for the $p$-Laplacian and applications, {\it Nonlinear Anal.} {\bf 32} (1998), 819-830.

\bibitem{2} F. Faraci, D. Motreanu, D. Puglisi, Positive solutions of quasilinear elliptic equations with dependence on the gradient, {\it Calc. Var.} {\bf 54} (2015), 525-538.

\bibitem{3} D. de Figueiredo, M. Girardi, M. Matzeu, Semilinear elliptic equations with dependence on the gradient via mountain-pass techniques, {\it Diff. Integral Equations} {\bf 17} (2004), 119-126.

\bibitem{4} L. Gasinski, N.S. Papageorgiou, {\it Nonlinear Analysis}, Chapman \& Hall/CRC, Boca Raton, FL, 2006.

\bibitem{5} L. Gasinski, N.S. Papageorgiou, Positive solutions for nonlinear elliptic problems with dependence on the gradient, {\it J. Differential Equations} {\bf 263} (2017), 1451-1476.

\bibitem{6} M. Girardi, M. Matzeu, Positive and negative solutions of a quasilinear elliptic equation by a mountain pass method and truncature techniques, {\it Nonlinear Anal.} {\bf 59} (2004), 199-210.

\bibitem{7} N.B. Huy, B.T. Quan, N.H. Khanh, Existence and multiplicity results for generalized logistic equations, {\it Nonlinear Anal.} {\bf 144} (2016), 77-92.

\bibitem{8} G. Lieberman, Boundary regularity for solutions of degenerate elliptic equations, {\it Nonlinear Anal.} {\bf 12} (1988), 1203-1219.

\bibitem{9} D. Mugnai, N.S. Papageorgiou, Resonant nonlinear Neumann problems with indefinite weight, {\it Ann. Sc. Norm. Super. Pisa Cl. Sci.} (5) {\bf 11} (2012), no. 4, 729-788.

\bibitem{10} N.S. Papageorgiou, V.D. R\u adulescu, Multiple solutions with precise sign information for nonlinear parametric Robin problems, {\it J. Differential Equations} {\bf 256} (2014), 393-430.

\bibitem{11} N.S. Papageorgiou, V.D. R\u adulescu, Nonlinear nonhomogeneous Robin problems with superlinear reaction term, {\it Adv. Nonlinear Studies} {\bf 16} (2016), 737-764.

\bibitem{12} N.S. Papageorgiou, V.D. R\u adulescu, D.D. Repov\v{s},  Nonlinear elliptic inclusions with unilateral constraint and dependence on the gradient, {\it Appl. Math. Optim.},  to appear (DOI: 10.1007/s00245-016-9392-y).

\bibitem{13} D. Ruiz, A priori estimates and existence of positive solutions for strongly nonlinear problems, {\it   J. Differential Equations} {\bf 199} (2004), 96-114.

\bibitem{14} P. Winkert, $L^{\infty}$-estimates for nonlinear elliptic Neumann boundary value problems, {\it Nonlin. Diff. Equations Appl. (NoDEA)} {\bf 17} (2010), 289-302.


\end{thebibliography}
\end{document}